\documentclass[12pt]{amsart} 
\usepackage{amssymb,amsmath,latexsym,enumerate,graphicx,bbm,mathptmx,microtype}
\newtheorem{theorem}{Theorem} 
\newtheorem{corollary}[theorem]{Corollary}

\newtheorem{lemma}[theorem]{Lemma}

\newtheorem*{rem}{Remarks}

\newtheorem*{exam}{Examples}

\newcommand\fl[1]{\left\lfloor {#1} \right\rfloor} 
\newcommand\saw[1]{\left(\!\left( #1 \right)\!\right)}
\newcommand\highlight[1]{\emph{#1}}

\def\Z{\mathbbm{Z}}
\def\N{\mathbbm{N}}
\def\R{\mathbbm{R}}

\def\B{{\overline B}}

\def\s{s}

\title{Bernoulli--Dedekind Sums}

\author{Matthias Beck}
\address{Department of Mathematics\\
         San Francisco State University\\
         San Francisco, CA 94132\\
         U.S.A.}
\email{beck@math.sfsu.edu, anastasia.chavez@gmail.com}

\author{Anastasia Chavez}

\thanks{
Research of M.\ Beck partially supported by NSF (DMS-0810105).
Research of A.\ Chavez partially supported by NSF (HRD-0331537) and NIH (RISE 5R25-6M059298-10).
We thank Mejid Bayad for stimulating discussions about a previous version of this paper.
}
\date{30 July 2010}

\begin{document}

\begin{abstract}
Let $p_1,p_2,\dots,p_n, a_1,a_2,\dots,a_n \in \N$, $x_1,x_2,\dots,x_n \in \R$, and denote the $k$th periodized Bernoulli polynomial by $\B_k(x)$.
We study expressions of the form
\[
  \sum_{h \bmod{a_k}} \ \prod_{\substack{i=1\\ i\not=k}}^{n} \ \B_{p_i}\left(a_i \frac{h+x_k}{a_k}-x_i\right).
\]
These \highlight{Bernoulli--Dedekind sums} generalize and unify various arithmetic sums introduced by
Dedekind, Apostol, Carlitz, Rademacher, Sczech, Hall--Wilson--Zagier, and others.
Generalized Dedekind sums appear in various areas such as analytic and algebraic number theory,
topology, algebraic and combinatorial geometry, and algorithmic complexity.
We exhibit a reciprocity theorem for the Bernoulli--Dedekind sums, which gives a unifying picture through a simple combinatorial proof.
\end{abstract}

\maketitle


\section{Introduction}

While studying the transformation properties of   
\[ \eta (z) := e^{ \frac{\pi i z}{12} } \prod_{ n \geq 1 } \left( 1 - e^{ 2 \pi i n z } \right) , \]
under $ \mbox{SL}_{2} ( \Z ) $, Dedekind, in the 1880's \cite{dedekind}, naturally arrived at the following arithmetic function:
Let $ \saw x $ be the sawtooth function defined by
  \begin{equation}\label{doublebrackets} \saw{x} := \left\{ \begin{array}{cl} \{ x \} - \frac{ 1 }{ 2 } & \mbox{ if }
 x \not\in \Z, \\
                                                                            0                         & \mbox{ if }
 x \in \Z . \end{array} \right. \end{equation}
Here $ \{ x \} = x - \fl x $ denotes the fractional part of $x$. 
For $a, b \in \N := \left\{ n \in \Z : n > 0 \right\} $, we define the \highlight{Dedekind sum} as
  \begin{equation}\label{dedsaw}
  \s (a,b) := \sum_{ k \bmod b } \saw{ \frac{ ka }{ b } } \saw{ \frac{ k }{ b } } . \end{equation}
The Dedekind sum and its generalizations have since intrigued mathematicians from various areas such as analytic \cite{almkvist,dedekind,dieterdedekind} and algebraic number theory \cite{meyerdedekind,solomondedsum}, topology \cite{hirzebruchzagier,meyersczech,zagier}, algebraic \cite{brionvergne,pommersheim,urzua} and combinatorial geometry \cite{ccd,mordell}, and algorithmic complexity \cite{knuth}. 

By means of the discrete Fourier series of the sawtooth function (see, for example, \cite[p.~14]{grosswald}),
it is not hard to write the Dedekind sum in terms of cotangents:
  \begin{equation}\label{dedcot} \s(a,b) = \frac{ 1 }{ 4b } \sum_{ k=1 }^{ b-1 } \cot \frac{ \pi k a }{ b } \cot \frac{ \pi k }{ b }\, .  \end{equation}
Starting with these two representations (\ref{dedsaw}) and (\ref{dedcot}) of $ \s (a, b) $, various generalizations of the Dedekind sum have been introduced. 
A previous paper \cite{cotangent} by the first author attempted to unify generalizations of the Dedekind sum in its `cotangent representation' (\ref{dedcot}). The present paper complements this by introducing a vast generalization of the `sawtooth representation' (\ref{dedsaw}) of the Dedekind sum.
The sawtooth function is the first periodized Bernoulli polynomial $\B_1 (u)$, where the \highlight{Bernoulli polynomial} $B_k(u)$ is defined, as usual, through
\[
  \frac{e^{uz}}{e^z-1}
  =\sum_{k\geq 0}\frac{B_k(u)}{k!} \, z^{k-1} ,
\]
and its periodized counterpart $\B_k(u)$ is defined as the unique function that is periodic with period 1 and coincides with $B_k(u)$ on $[0,1)$, except that we set $\B_1(u)=0$ for $u\in \Z$.
To define our generalization, let $p_1,p_2,\dots,p_n, a_1,a_2,\dots,a_n \in \N$, $x_1,x_2,\dots,x_n \in \R$, and
\[
  A_k := \left(a_1,a_2,\dots,\widehat{a_k},\dots,a_n\right), \qquad 
  X_k:=\left(x_1,x_2,\dots,\widehat{x_k},\dots,x_n\right) , \qquad
  P_k := \left(p_1,p_2,\dots,\widehat{p_k},\dots,p_n\right) ,
\]
where $\widehat{a}_k$ means we omit the entry $a_k$. Then we define the \highlight{Bernoulli--Dedekind sum} as 
\begin{equation}\label{bdsum}
 S_{P_{k}}\left(\begin{matrix} A_k & a_k \\ X_k & x_k \end{matrix} \right)
 := \sum_{h \bmod{a_k}} \ \prod_{\substack{i=1\\ i\not=k}}^{n} \ \B_{p_i}\left(a_i \frac{h+x_k}{a_k}-x_i\right).
\end{equation}
The Bernoulli--Dedekind sums include as special cases various previously-defined Dedekind-like sums, which we will discuss in detail in Section~\ref{history}.

The most fundamental and important theorems for any of the generalized De\-de\-kind sums are the
{\it reciprocity laws}: an appropriate sum of generalized Dedekind sums (usually permuting the
arguments in a cyclic fashion) gives a simple rational expression. 
The famous reciprocity law for the classical Dedekind sum is as old as the sum itself: 
\begin{theorem}[Dedekind]\label{dedreclaw}
If $a,b \in \N$ are coprime then
  \[ \s (a,b) + \s (b,a) = - \frac{1}{4} + \frac{1}{12} \left( \frac{ a }{ b } + \frac{ 1 }{ ab } + \frac{ b }{ a }
 \right) \ . \]
\end{theorem}

Our main goal in this paper is to prove a reciprocity theorem for the Bernoulli--Dedekind sums, which is most conveniently stated in terms of generating functions. For nonzero variables $y_1,y_2,\dots,y_n$, let $Y_k:=\left(y_1,y_2,\dots,\widehat{y_k},\dots,y_n\right)$ and
\[
  \Omega \left(\begin{matrix} A_k & a_k \\ X_k & x_k \\ Y_k & \end{matrix} \right)
  := \sum_{ \left( p_1, \dots, \widehat{p_k}, \dots, p_n \right) \in \Z_{ \ge 0 }^{ n-1 } } \frac{1}{p_1!p_2!\cdots p_{k-1}! p_{k+1}! \cdots p_n!} \ S_{P_{k}}\left(\begin{matrix} A_k & a_k \\ X_k & x_k \end{matrix} \right)\prod_{\substack{i=1\\ i\not=k}}^n \left(\frac{y_i}{a_i}\right)^{p_i-1}.
\]
Our main result is as follows.

\begin{theorem}\label{multivariatereciprocitytheorem}
Let $a_1,a_2,\dots,a_n \in \N$ be pairwise coprime, $x_1,x_2,\dots,x_n \in \R$, and let $y_1,y_2,\dots,y_n$ be nonzero variables such that $y_1+y_2+\cdots+y_n=0$.
If $\frac{x_u-h_u}{a_u}-\frac{x_v-h_v}{a_v}\not\in\Z$ whenever $1\leq u<v \leq n$ and $h_u,h_v\in\Z$, then
\[
  \sum_{k=1}^{n}\Omega\left(\begin{matrix} A_k & a_k \\ X_k & x_k \\ Y_k & \end{matrix}\right)=0 \, .
\]
\end{theorem}

The next section will illustrate the span from \eqref{dedsaw} to \eqref{bdsum} (and from Theorem \ref{dedreclaw} to Theorem \ref{multivariatereciprocitytheorem}).
Section \ref{proofsection} is devoted to the proof of Theorem~\ref{multivariatereciprocitytheorem}.
As an interlude, we exhibit in Section \ref{petknoppsection} a Petersson--Knopp identity \cite{knoppdedekind} for the Bernoulli--Dedekind sum.
In Section \ref{hwzsection}, we show that---within limits---our ideas can also be applied to more general reciprocity theorems, namely versions of Theorem \ref{multivariatereciprocitytheorem} in which the condition $\frac{x_u-h_u}{a_u}-\frac{x_v-h_v}{a_v}\not\in\Z$ can be omitted.


\section{Various Dedekind-like sums}\label{history} 
In this section we will give an overview of previously-defined generalizations of the Dedekind sum (rather, of its `sawtooth representation' (\ref{dedsaw})). 
We do not claim any completeness but hope to give some picture of what has been introduced in the past. 

Apostol \cite{apostoldedekind} replaced one of the sawtooth functions in (\ref{dedsaw}) by an arbitrary Bernoulli function:
  \begin{equation}\label{apo} \sum_{ k \bmod b } \saw{\frac{k}{b}} \B_{n} \left( \frac{ka }{b } 
\right) .  \end{equation}
Apostol's idea was generalized by Carlitz \cite{carlitzbernoulli} and Mikol\'as \cite{mikolas} 
to
  \[ \sum_{ k \bmod b } \B_{m} \left( \frac{ kb }{ a } \right) \B_{n} \left( \frac{ kc }{ a } \right) . \] 
Another way of generalizing (\ref{dedsaw}) is to shift the argument of the sawtooth functions.
This was introduced by Meyer \cite{meyerdedekind} and Dieter \cite{dieterdedekind}, and brought to a solid
ground by Rademacher \cite{rademacherdedekind}: 
For $ a,b \in \N $, $ x,y \in \R $, the \highlight{Dedekind--Rademacher sum} is defined by
  \[ \sum_{ k \bmod b } \saw{ a \frac{ k+y }{ b } - x } \saw{ \frac{ k+y }{ b } } \ . \]
The ideas of Apostol and Rademacher can also be combined: Tak\'acs \cite{takacs} introduced a shift in Apostol's sum (\ref{apo}):
  \[ \sum_{ k \bmod b } \saw{ \frac{k+y}{b} } \B_{n} \left( a \frac{k+y}{b } -
 x \right) \ . \]
This was further generalized by Halbritter \cite{halbritter} and later by Hall, Wilson, and Zagier \cite{hallwilsonzagier} to 
  \[ \sum_{h \bmod{c}}\B_m\left(a \frac{h+z}{c}-x\right)\B_n\left(b\frac{h+z}{c}-y\right) , \] 
where $ a,b,c, m, n \in \N $ and $ x,y,z \in \R $.
The Hall--Wilson--Zagier sum, in turn, is a special case of our Bernoulli--Dedekind sum, namely, 
$
S_{m,n}\left(\begin{matrix}A_3 & a_3 \\ X_3 & x_3 \end{matrix}\right) ,
$
where $(a_1,a_2,a_3)=(a,b,c)$, $(x_1, x_2, x_3)=(x,y,z)$, and $P_3=(m,n)$ (and thus $A_3=\left(a,b\right)$ and $X_3=\left(x,y\right)$).

The central theorems of all of the above-cited papers are reciprocity theorems for each generalized Dedekind sum.
To give one example, we state the reciprocity theorem of~\cite{hallwilsonzagier}.

\begin{theorem}[Hall--Wilson--Zagier]\label{hwzreciprocity}
Let $a_1,a_2,a_3 \in \N$ be pairwise coprime, $x_1,x_2,x_3 \in \R$, and let $y_1,y_2,y_3$ be nonzero variables such that $y_1+y_2+y_3=0$. Then
\begin{align*}
  & \Omega\left(\begin{matrix} a_1 & a_2 & a_3 \\ x_1 & x_2 & x_3 \\ y_1 & y_2 & \end{matrix}\right)
  + \Omega\left(\begin{matrix} a_2 & a_3 & a_1 \\ x_2 & x_3 & x_1 \\ y_2 & y_3 & \end{matrix}\right)
  + \Omega\left(\begin{matrix} a_3 & a_1 & a_2 \\ x_3 & x_1 & x_2 \\ y_3 & y_1 & \end{matrix}\right) \\
  & \qquad
  = \begin{cases}
         -1/4     &\text{ if } (x_1,x_2,x_3)\in(a_1,a_2,a_3)\R+\Z^3, \\
         0        &\text{ otherwise. }
    \end{cases}
\end{align*}
\end{theorem}

As noted elsewhere, the statement of this theorem in \cite{hallwilsonzagier} missed the minus sign in front of $\frac 1 4$ (their proof is correct nevertheless).
We should also remark that it is a somewhat nontrivial (but fun) exercise to derive Dedekind's Theorem \ref{dedreclaw} from Hall--Wilson--Zagier's Theorem~\ref{hwzreciprocity}.

It is the generic (``otherwise") case of Theorem \ref{hwzreciprocity} that our Theorem \ref{multivariatereciprocitytheorem} extends.
The other case (which essentially deals with Bernoulli--Dededekind sums for which $x_1 = x_2 = \dots = x_n = 0$) was recently extended by Bayad and Raouj \cite{bayadraouj}. 
For Theorem \ref{hwzreciprocity}, there are essentially only these two ``extreme" cases; however, for higher-dimensional Bernoulli--Dedekind sums, there are more intermediate cases (in which we have a partial linear relation such as $(x_1,x_2,x_3)\in(a_1,a_2,a_3)\R+\Z^3$), and it is not clear to us how one can easily deal with them. We address this issue in Section~\ref{finalsection}.

It is worth mentioning Hu's thesis \cite{huthesis} which contains another variant of a `multidimensional' Hall--Wilson--Zagier sum. Hu's reciprocity theorem is equivalent to Theorem \ref{hwzreciprocity} for $n=3$ but Hu's generalized Dedekind sums and their reciprocity theorems become different from our Bernoulli--Dedekind sums for $n \ge 4$.

Finally, we note that both Theorems \ref{multivariatereciprocitytheorem} and \ref{hwzreciprocity} can be implicitly seen in the work of Sczech \cite{sczechellipticdedekind,sczecheisensteincocycles}. 
We consider our main contribution as giving a unifying picture and a simple combinatorial reciprocity proof for Bernoulli--Dedekind sums.


\section{Proof of Theorem \ref{multivariatereciprocitytheorem}}\label{proofsection}

We start our journey towards a proof of Theorem \ref{multivariatereciprocitytheorem} with the following lemma on fractional parts, whose easy proof is left to the reader.

\begin{lemma}\label{posdif}
Given $a,b,c\in \R$,
\begin{align*}
\{a-b\}-\{a-c\} \ge 0 \quad &\Rightarrow \quad \{a-b\}-\{a-c\}=\{ c-b\} \\
\{a-b\}-\{a-c\} \le 0 \quad &\Rightarrow \quad \{a-b\}-\{a-c\}=-\{ b-c\} \, .
\end{align*}
\end{lemma}

Almost as easy is the proof of the following well-known lemma \cite{raabe}.

\begin{lemma}[Raabe's formula]\label{Raabe's formula}
For $a\in \N$, $x\in \R$,
\[
\sum_{h \bmod{a}}\B_m\left(x+\frac{h}{a}\right)=a^{1-m} \ \B_m(ax) \, .
\]
\end{lemma}

Consequently, we can manipulate a Bernoulli--Dedekind sum as follows.

\begin{align}
  S_{P_{k}}\left(\begin{matrix} A_k & a_k \\ X_k & x_k \end{matrix} \right) \ \prod_{\substack{j=1\\ j\not=k}}^n a_j^{1-m_j} 
  &= \sum_{h \bmod{a_k}} \ \prod_{\substack{j=1\\ j\not=k}}^{n} \ \B_{p_j}\left(a_j \frac{h+x_k}{a_k}-x_j\right) a_j^{1-m_j} \notag \\
  &=\sum_{ \substack{ h_1 \bmod{a_1} \\ \cdots \\ h_n \bmod{a_n} } } \ \prod_{\substack{j=1\\ j\not=k}}^{n} \ \B_{p_j}\left(\frac{x_k+h_k}{a_k}-\frac{x_j+h_j}{a_j}\right),\label{bdsumusingraabe}
\end{align}
where the sum includes the original summand $h$ we now call $h_k$.
Introducing the short-hand $r_j:=\frac{x_j+h_j}{a_j}$, \eqref{bdsumusingraabe} gives
\begin{align}
  \Omega \left(\begin{matrix} A_k & a_k \\ X_k & x_k \\ Y_k & \end{matrix} \right)
  &= \sum_{ \left( p_1, \dots, \widehat{p_k}, \dots, p_n \right) \in \Z_{ \ge 0 }^{ n-1 } } \frac{1}{p_1!p_2!\cdots p_{k-1}!p_{k+1}!\cdots p_n!} \ S_{P_{k}}\left(\begin{matrix} A_k & a_k \\ X_k & x_k \end{matrix} \right)\prod_{\substack{j=1\\ j\not=k}}^n \left(\frac{y_j}{a_j}\right)^{p_j-1} \notag \\
  &= \sum_{ \substack{ h_1 \bmod{a_1} \\ \cdots \\ h_n \bmod{a_n} } } \ \sum_{ \left( p_1, \dots, \widehat{p_k}, \dots, p_n \right) \in \Z_{ \ge 0 }^{ n-1 } } \frac{1}{p_1!p_2!\cdots p_{k-1}!p_{k+1}!\cdots p_n!}\ \prod_{\substack{j=1\\ j\not=k}}^{n} \ \B_{p_j}\left(r_k-r_j\right)\, y_{i}^{p_j-1} \notag \\
  &= \sum_{ \substack{ h_1 \bmod{a_1} \\ \cdots \\ h_n \bmod{a_n} } } \ \prod_{\substack{j=1\\ j\not=k}}^{n}\ \beta\left(r_k-r_j,y_j\right),\label{beta identity}
\end{align}
where
\[
  \beta(u,z)
  :=\sum_{k\geq 0}\frac{\B_k(u)}{k!} \, z^{k-1} .
\]
Note that
\begin{equation}\label{betalemm}
  \beta(u,z)=
  \begin{cases}
    \frac{1}{2} \, \frac{e^z+1}{e^z-1} &\text{ if } u \in \Z , \\
    \frac{e^{\{ u \} z}}{e^z-1} &\text{ if } u \notin \Z ,
  \end{cases}  
\end{equation}
and so it is clear that (\ref{beta identity}) depends on the differences $r_k-r_j$, and the $\beta\left(r_k-r_j,y_j\right)$ depend on whether or not these differences are integers.
This is the reason for our crucial assumption that $\frac{x_u-h_u}{a_u}-\frac{x_v-h_v}{a_v}\not\in\Z$ whenever $1\leq u<v \leq n$ and $h_u,h_v\in\Z$ in Theorem \ref{multivariatereciprocitytheorem}: it allows us to use the second case of \eqref{betalemm} throughout.
For the rest of this section, we assume the differences $r_k-r_j$ are not integers.
Then
\begin{align*}
  \sum_{k=1}^{n}\Omega\left(\begin{matrix} A_k & a_k \\ X_k & x_k \\ Y_k & \end{matrix} \right)
  &= \sum_{k=1}^{n} \ \sum_{ \substack{ h_1 \bmod{a_1} \\ \cdots \\ h_n \bmod{a_n} } } \ \prod_{\substack{j=1\\ j\not=k}}^{n} \ \beta\left(r_k-r_j,y_j\right)
  = \sum_{ \substack{ h_1 \bmod{a_1} \\ \cdots \\ h_n \bmod{a_n} } } \sum_{k=1}^{n} \ \prod_{\substack{j=1\\ j \ne k}}^{n}\frac{e^{\{r_k-r_j\}y_j}}{e^{y_j}-1} \, \frac{e^{y_k}-1}{e^{y_k}-1} \\
  &= \sum_{ \substack{ h_1 \bmod{a_1} \\ \cdots \\ h_n \bmod{a_n} } } \frac{ \sum_{k=1}^{n} e^{ y_k } \prod_{\substack{j=1\\ j \ne k}}^{n}e^{\{r_k-r_j\}y_j} - \sum_{k=1}^{n} \prod_{\substack{j=1\\ j \ne k}}^{n}e^{\{r_k-r_j\}y_j} }{\prod_{j=1}^{n}\left(e^{y_j}-1\right)} \, .
\end{align*}
Theorem \ref{multivariatereciprocitytheorem} will be proved once we can show that the numerator 
\[
  N := \sum_{k=1}^{n} e^{ y_k } \prod_{\substack{j=1\\ j \ne k}}^{n}e^{\{r_k-r_j\}y_j} - \sum_{k=1}^{n} \prod_{\substack{j=1\\ j \ne k}}^{n}e^{\{r_k-r_j\}y_j}
\]
in this expression vanishes.
We separate the $k=n$ terms and use the assumption $y_1+y_2+\cdots+y_n=0$:
\begin{align*}
  N &= \sum_{k=1}^{n-1}e^{\{r_k-r_n\}y_n+y_k}\prod_{\substack{j=1\\ j \ne k}}^{n-1}e^{\{r_k-r_j\}y_j}-\sum_{k=1}^{n-1}e^{\{r_k-r_n\}y_n}\prod_{\substack{j=1\\ j \ne k}}^{n-1}e^{\{r_k-r_j\}y_j} + e^{ y_n } \prod_{j=1}^{n-1}e^{\{r_n-r_j\} y_j } - \prod_{j=1}^{n-1}e^{\{r_n-r_j\}y_j}\notag\\
  &= \sum_{k=1}^{n-1}e^{\{r_k-r_n\}\left(-y_1-\cdots-y_{n-1}\right)+y_k}\prod_{\substack{j=1\\ j \ne k}}^{n-1}e^{\{r_k-r_j\}y_j} - \sum_{k=1}^{n-1}e^{\{r_k-r_n\}\left(-y_1-\cdots-y_{n-1}\right)}\prod_{\substack{j=1\\ j \ne k}}^{n-1} e^{\{r_k-r_j\}y_j} \\
  &\qquad + e^{ -y_1-\cdots-y_{n-1} } \prod_{j=1}^{n-1}e^{\{r_n-r_j\}y_j } - \prod_{j=1}^{n-1}e^{\{r_n-r_j\}y_j}.\notag
\end{align*}
Note that $1-\{r_k-r_n\}=\{r_n-r_k\}$ since $r_k-r_n\not\in \Z$. Thus
\begin{align}
  N &= \sum_{k=1}^{n-1}e^{\{r_n-r_k\}y_k}\prod_{\substack{j=1\\ j \ne k}}^{n-1}e^{\left(\{r_k-r_j\}-\{r_k-r_n\}\right)y_j} - \sum_{k=1}^{n-1}e^{-\{r_k-r_n\}y_k}\prod_{\substack{j=1\\ j \ne k}}^{n-1}e^{\left(\{r_k-r_j\}-\{r_k-r_n\}\right)y_j} \label{numerator1} \\
  &\qquad + \prod_{j=1}^{n-1}e^{-\{r_j-r_n\}y_j}  - \prod_{j=1}^{n-1}e^{\{r_n-r_j\}y_j}. \notag
\end{align}
\def\C{\sigma}
We will show that we can find identical pairs of exponents in terms with opposite signs in this expression, and so the sum vanishes. 
Only three types of exponents appear in (\ref{numerator1}):
\[
  \{r_n-r_k\} y_k \, , \qquad -\{r_k-r_n\} y_k \, , \qquad \text{ and } \qquad \left( \{r_k-r_j\}-\{r_k-r_n\} \right) y_j \, .
\]
By Lemma \ref{posdif}, $\left( \{r_k-r_j\}-\{r_k-r_n\} \right) y_j$ equals $\{r_n-r_j\} y_j$ or $-\{r_j-r_n\} y_j$, and so the exponents can be condensed to just the first two types.
Moreover, the sign of $\{r_k-r_j\}-\{r_k-r_n\}$ determines if it is equal to $\{r_n-r_j\}$ or  $-\{r_j-r_n\}$, and so all the information of $N$ can be encoded by a sign matrix: the term $\{r_n-r_k\} y_k$ gets encoded by $+$, the term $-\{r_j-r_n\} y_j$ by $-$, and the term $\left( \{r_k-r_j\}-\{r_k-r_n\} \right) y_j$ by the sign of $\{r_k-r_j\}-\{r_k-r_n\}$, which we denote by $\C_{ kj }$.
For example, the exponent corresponding to $k=1$ in the first sum
\begin{equation}\label{firstsumeq}
  \sum_{k=1}^{n-1}e^{\{r_n-r_k\}y_k}\prod_{\substack{j=1\\ j \ne k}}^{n-1}e^{\left(\{r_k-r_j\}-\{r_k-r_n\}\right)y_j}
\end{equation}
is 
\[
  \{r_n-r_1\}y_1 +\left(\{r_1-r_2\}-\{r_1-r_n\}\right)y_2+\cdots+\left(\{r_1-r_{n-1}\}-\{r_1-r_n\}\right)y_{n-1}
\]
and gets represented by the sign vector $\left(+, \C_{12},\dots,\C_{1, \, n-1}\right)$.
More generally, the $k$th term in \eqref{firstsumeq} is represented by the sign vector $\left( \C_{ k,1 } , \dots, \C_{ k,k-1 } , + , \C_{ k, k+1 } , \dots, \C_{ k, n-1 } \right)$.
Similarly, the $k$th term in the second sum of \eqref{numerator1},
\[
  - \sum_{k=1}^{n-1}e^{-\{r_k-r_n\}y_k}\prod_{\substack{j=1\\ j \ne k}}^{n-1}e^{\left(\{r_k-r_j\}-\{r_k-r_n\}\right)y_j} ,
\]
is represented by the sign vector $\left( \C_{ k,1 } , \dots, \C_{ k,k-1 } , - , \C_{ k, k+1 } , \dots, \C_{ k, n-1 } \right)$.
The two final terms in \eqref{numerator1},
\[
  \prod_{j=1}^{n-1}e^{-\{r_j-r_n\}y_j} - \prod_{j=1}^{n-1}e^{\{r_n-r_j\}y_j} ,
\]
are represented by the respective sign vectors $\left(- ,- ,\dots,-\right)$ and $\left(+, +,\dots,+\right)$.

\def\mp{M_{ \rm pos }}
\def\mn{M_{ \rm neg }}

Let $\mp$, resp.\ $\mn$, be the matrix of all sign vectors representing the exponents of the positive, resp.\ negative, terms of $N$, where we place the sign vector representing, e.g., an exponent from the $k$th positive term in the $k$th row of matrix $\mp$.
Thus we have constructed the matrices
\[
  \mp = \left(\begin{matrix} + & \C_{1 2} & \C_{1 3} & \cdots & \C_{1, n-1}\\ \C_{2 1} & + & \C_{2 3} & \cdots & \C_{2, n-1} \\ \C_{3 1} & \C_{3 2} & + & \cdots & \C_{3, n-1}\\ \vdots & \vdots & \vdots & \ddots & \vdots \\ \C_{n-1, 1} & \C_{n-1, 2} & \C_{n-1, 3} &  \cdots & +\\ - & - & - & \cdots & - \end{matrix}\right)
\]
and
\[
  \mn = \left(\begin{matrix} - & \C_{1 2} & \C_{1 3} & \cdots & \C_{1, n-1}\\ \C_{2 1} & - & \C_{2 3} & \cdots & \C_{2, n-1}  \\ \C_{3 1} & \C_{3 2} & - & \cdots & \C_{3, n-1}\\ \vdots & \vdots & \vdots & \ddots & \vdots \\ \C_{n-1, 1} & \C_{n-1, 2} & \C_{n-1, 3} &  \cdots & -\\ + & + & + & \cdots & + \end{matrix}\right) ,
\]
and our goal $N=0$ will follow from proving that $\mp = \mn$ after row swapping.
To show the latter, we first collect some properties of $\mp$ and $\mn$.

\begin{lemma}\label{lemmapottpori} $ $ 
\begin{enumerate}[{\rm (a)}]
\item\label{diagonalrows} The matrix $\mp$ has $+$ entries on the diagonal and the last row consists entirely of $-$ entries; $\mn$ has $-$ entries on the diagonal and the last row consists entirely of $+$ entries.
\item\label{posequalsneglemma} $\C_{i j}=+$ if and only if $\C_{j i}=-$.
\item\label{matching} If $\C_{i j}=+$ and $\C_{i k}=-$ then $\C_{j k}=-$.
\end{enumerate}
\end{lemma}

\begin{proof}
(\ref{diagonalrows}) follows by construction.

\vspace{5pt}
\noindent
(\ref{posequalsneglemma}) follows with Lemma \ref{posdif}.

\vspace{5pt}
\noindent
(\ref{matching}) Assume $\C_{i j}=+$ and $\C_{i k}=-$. Then $\{r_i-r_j\}-\{r_i-r_n\}>0$ and $\{r_i-r_k\}-\{r_i-r_n\}<0$, and by Lemma \ref{posdif}
\begin{equation}\label{dif1}
  \{r_i-r_j\}-\{r_i-r_n\}=\{r_n-r_j\}
\end{equation}
and
\begin{equation}\label{dif2}
  \{r_i-r_k\}-\{r_i-r_n\}=-\{r_k-r_i\}.
\end{equation}
Then the difference (\ref{dif1})$-$(\ref{dif2}) is positive and we get
\[
  \{r_i-r_j\}-\{r_i-r_k\}=\{r_n-r_j\}+\{r_k-r_i\} \, .
\]
The final identity is positive, which means the left-hand side is positive. Then by Lemma \ref{posdif}
\[
  \{r_i-r_j\}-\{r_i-r_k\}=\{r_k-r_j\} \, ,
\]
and so
\[
  \{r_k-r_j\}=\{r_n-r_j\}+\{r_k-r_i\}.
\]
But then $\C_{j k}=-$ follows from
\begin{align*}
\{r_j-r_k\}-\{r_j-r_n\}&=\{r_j-r_k\}-1+1-\{r_j-r_n\} \\
&=-\{r_k-r_j\}+\{r_n-r_j\} \\
&=-\{r_k-r_i\} \, . \qedhere
\end{align*}
\end{proof}

Part (\ref{posequalsneglemma}) of this lemma allows us to update the sign matrices:
\[
  \mp=\left(\begin{matrix} + & \C_{1 2} & \C_{1 3} & \cdots & \C_{1, n-1}\\ -\C_{1 2} & + & \C_{2 3} & \cdots & \C_{2, n-1} \\ -\C_{1 3} & -\C_{2 3} & + & \cdots & \C_{3, n-1}\\ \vdots & \vdots & \vdots & \ddots & \vdots \\-\C_{1, n-1} & -\C_{2, n-1} & -\C_{3, n-1} &  \cdots & +\\ - & - & - & \cdots & - \end{matrix}\right)
\]
and
\[
  \mn=\left(\begin{matrix} - & \C_{1 2} & \C_{1 3} & \cdots & \C_{1, n-1}\\ -\C_{1 2} & - & \C_{2 3} & \cdots & \C_{2, n-1}  \\ -\C_{1 3} & -\C_{2 3} & - & \cdots & \C_{3, n-1}\\ \vdots & \vdots & \vdots & \ddots & \vdots \\-\C_{1, n-1} & -\C_{2, n-1} & -\C_{3, n-1} &  \cdots & -\\ + & + & + & \cdots & + \end{matrix}\right) .
\]

\begin{lemma}\label{uniquerows}
Each of the matrices $\mp$ and $\mn$ has a unique row with $k$ $+$'s, for each $0\leq  k \leq n-1$.
\end{lemma}

\begin{proof}
We will prove this for $\mp$; the statement for $\mn$ follows then immediately.

We begin by showing that every row of the matrix $\mp$ is unique. Suppose on the contrary that row $m$ and row $l$ of $\mp$ are equal. Then these rows look as follows:
\[
  \begin{matrix}\text{ row $m$:}\hspace{.4cm}& -\C_{1 m} & -\C_{2 m} & \cdots & + & \cdots & \C_{m l} & \cdots & \C_{m, n-1}\\
          	      \text{ row $l$:}\hspace{.4cm}  & -\C_{1 l} & -\C_{2 l} & \cdots & -\C_{m l} &\cdots & + & \cdots & \C_{l, n-1} \, .
  \end{matrix}
\]
Then $\C_{m l}=+$ and $-\C_{m l}=+$, which contradicts Lemma \ref{lemmapottpori}(\ref{posequalsneglemma}).

Next we will show that no two rows contain the same number of $+$'s. Suppose on the contrary that row $m$ and row $l$ of $\mp$ contain the same number of $+$'s (and are not equal). 

Assume $\C_{m l}=+$. Since the $m$th row does not entirely consist of $+$'s, there exists a $-$ in column, say, $k$. Then by Lemma \ref{lemmapottpori}(\ref{matching}), the entry $\C_{k l}$ is $-$. So, for every $-$ in row $m$, Lemma \ref{lemmapottpori}(\ref{matching}) can be applied to show there is a $-$ in the same column entry of row $l$. But $-\C_{m l}=-$, and so row $l$ contains at least one more $-$ than row $m$, a contradiction.
If, on the other hand, $\C_{m l}=-$, then we can repeat the above argument for row $l$, starting with the entry $-\C_{m l}=+$. 

We have shown that no two rows contain the same number of $+$'s and so, for each $0\leq  k \leq n-1$, there exists a unique row with $k$ $+$'s.
\end{proof}

Lemma \ref{uniquerows} allows us to match up the unique rows of $k$ $+$'s in $\mp$ and $\mn$. Thus, after row swapping, $\mp = \mn$, which, by our previous argument, finally proves Theorem~\ref{multivariatereciprocitytheorem}.


\section{Petersson--Knopp Identities}\label{petknoppsection}

Another basic identity on the classical Dedekind sum is the following \cite{knoppdedekind}.
\begin{theorem}[Petersson--Knopp]\label{petknopp} Suppose $a,b \in \N$ are coprime. Then
\[
  \sum_{ d|m } \ \sum_{ k \mbox{ \rm \scriptsize mod } d } \s \left( \frac{ m }{ d } b + ka , ad \right) = \sigma (m) \, \s ( b,a ) \, .
\]
Here $ \sigma (m) $ denotes the sum of the positive divisors of $m$.
\end{theorem}
This result has been extended to certain generalized Dedekind sums \cite{apostolvu,parsonrosen,zheng} and takes its most general form \cite{cotangent} for sums 
\[ 
  S \left( a; a_{1}, \dots , a_{n} \right) := \sum_{ k \mbox{ \rm \scriptsize mod } a } f_{1} \left( \frac{k a_{1}  }{a } \right) \cdots f_{n} \left( \frac{k a_{n}  }{a } \right)
\]
{\bf of Dedekind type with weight $ \left( m_{1} , \dots , m_{n} \right) $}, i.e., when for all $ j = 1, \dots, n $, $ f_{j} (x+1) = f_{j} (x) $ and for all $ a \in \N$,
\begin{equation}\label{weight}
  \sum_{ k \mbox{ \rm \scriptsize mod } a } f_{j} \left( x + \frac{ k }{ a } \right) = a^{ m_{j} } f_{j} (ax) \ .
\end{equation}
Note that the Bernoulli functions $ \B_{k} (x) $ satisfy (\ref{weight}) (with `weight' $ 1-k $), due to Lemma \ref{Raabe's formula}.
The following extension of Theorem \ref{petknopp} was proved in \cite{cotangent}.
\begin{theorem}\label{petknoppcot} Let $a, a_{1}, \dots , a_{n} \in \N $. If
  \[ S \left( a; a_{1}, \dots , a_{n} \right) := \sum_{ k \mbox{ \rm \scriptsize mod } a } f_{1} \left( \frac{k a_{1}  }{a } \right) \cdots f_{n} \left( \frac{k a_{n}  }{a } \right)  \]
is of Dedekind type with weight $ \left( m_{1} , \dots , m_{n} \right) $ then
\begin{align*}
  &\sum_{ d | m } d^{ - m_{1} - \dots - m_{n} } \! \sum_{ r_{1} , \dots , r_{n} \mbox{ \rm \scriptsize mod } d } \! S \left( ad; \frac{ m }{ d } a_{1} + r_{1} a , \dots , \frac{ m }{ d } a_{n} + r_{n} \, a  \right) \\
  &\qquad = m \, \sigma_{ n - 1 - m_{1} - \dots - m_{n} } (m) \, S \left( a; a_{1}, \dots , a_{n} \right) .
\end{align*}
\end{theorem}

This theorem together with Lemma \ref{Raabe's formula} immediately gives a Petersson--Knopp identity for the Bernoulli--Dedekind sums
\[
 S_{p_1, \dots, p_n}\left(\begin{matrix} (a_1, \dots, a_n) & a_0 \\ (0, \dots, 0) & 0 \end{matrix} \right)
 = \sum_{h \bmod{a_0}} \ \prod_{i=1}^n \ \B_{p_i}\left(a_i \frac{h}{a_0} \right).
\]

\begin{corollary}
\begin{align*}
  &\sum_{ d | m } d^{ p_{1} + \dots + p_{n} - n + 1 } \! \sum_{ r_{1} , \dots , r_{n} \mbox{ \rm \scriptsize mod } d } S_{p_1, \dots, p_n} \left(\begin{matrix} (\frac{ m }{ d } a_{1} + r_{1} \, a_0, \dots, \frac{ m }{ d } a_{n} + r_{n} \, a_0) & a_0 \\ (0, \dots, 0) & 0 \end{matrix} \right) \\
  &\qquad = m \, \sigma_{ p_{1} + \dots + p_{n} } (m) \ S_{p_1, \dots, p_n}\left(\begin{matrix} (a_1, \dots, a_n) & a_0 \\ (0, \dots, 0) & 0 \end{matrix} \right) .
\end{align*}
\end{corollary}


\section{Hall--Wilson--Zagier Revisited}\label{hwzsection}

In this section, we show how our ideas can be used to prove Hall--Wilson--Zagier's Theorem \ref{hwzreciprocity}.

\begin{proof}[Proof of Theorem \ref{hwzreciprocity}]
We want to show that
\begin{align*}
  & \Omega\left(\begin{matrix} a_1 & a_2 & a_3 \\ x_1 & x_2 & x_3 \\ y_1 & y_2 & \end{matrix}\right)
  + \Omega\left(\begin{matrix} a_2 & a_3 & a_1 \\ x_2 & x_3 & x_1 \\ y_2 & y_3 & \end{matrix}\right)
  + \Omega\left(\begin{matrix} a_3 & a_1 & a_2 \\ x_3 & x_1 & x_2 \\ y_3 & y_1 & \end{matrix}\right) \\
  & \qquad
  = \begin{cases}
         -1/4     &\text{ if } (x_1,x_2,x_3)\in(a_1,a_2,a_3)\R+\Z^3, \\
         0        &\text{ otherwise. }
    \end{cases}
\end{align*}
By \eqref{beta identity},
\[
  \sum_{ k=1 }^3 \Omega \left(\begin{matrix} A_k & a_k \\ X_k & x_k \\ Y_k & \end{matrix} \right)
  = \sum_{ k=1 }^3 \sum_{ \substack{ h_1 \bmod{a_1} \\ h_2 \bmod{a_2} \\ h_3 \bmod{a_3} } } \ \prod_{\substack{i=1\\ i\not=k}}^3\ \beta\left(r_k-r_i,y_i\right) ,
\]
where $r_i=\frac{x_i+h_i}{a_i}$.
We have to examine the following cases:
\begin{enumerate}[(i)]
\item $(x_1,x_2,x_3)\in(a_1,a_2,a_3)\R+\Z^3$;
\item $(x_i,x_j)\in(a_i,a_j)\R+\Z^2$ for some $1\leq i <j \leq 3$ but not (i);
\item none of the above.
\end{enumerate}

\vspace{6pt}
\noindent
Case (iii) is covered by Theorem~\ref{multivariatereciprocitytheorem}.

\vspace{6pt}
\noindent
Case (i). We have $x_i=\lambda a_i + z_i$ for each $i$, for some $\lambda\in\R$ and $z_i\in\Z$. Thus 
\[
  r_i - r_j
  = \frac{h_i+\lambda a_i + z_i}{a_i}-\frac{h_j+\lambda a_j + z_j}{a_j}
  = \frac{h_i + z_i}{a_i}-\frac{h_j + z_j}{a_j}
\]
and so the $z_i$'s simply permute the indices $h_i$. But since each $h_i$ gets summed over a complete residue system $\bmod{a_i}$,
\begin{align*}
  & \Omega\left(\begin{matrix} a_1 & a_2 & a_3 \\ x_1 & x_2 & x_3 \\ y_1 & y_2 & \end{matrix}\right)
  + \Omega\left(\begin{matrix} a_2 & a_3 & a_1 \\ x_2 & x_3 & x_1 \\ y_2 & y_3 & \end{matrix}\right)
  + \Omega\left(\begin{matrix} a_3 & a_1 & a_2 \\ x_3 & x_1 & x_2 \\ y_3 & y_1 & \end{matrix}\right) \\
  & \qquad = \sum_{ \substack{ h_1 \bmod{a_1} \\ h_2 \bmod{a_2} \\ h_3 \bmod{a_3} } } \sum_{k=1}^{3} \prod_{\substack{i=1\\ i\not=k}}^{3}\beta\left(\frac{h_k }{a_k}-\frac{h_i }{a_i},y_i\right) .
\end{align*}
Since $(a_i,a_j)=1$, $\frac{h_i }{a_i}-\frac{h_j }{a_j}\in\Z$ occurs only when $h_i=h_j=0$. Thus, we can split up the above sum
\begin{align*}
  & \Omega\left(\begin{matrix} a_1 & a_2 & a_3 \\ x_1 & x_2 & x_3 \\ y_1 & y_2 & \end{matrix}\right)
  + \Omega\left(\begin{matrix} a_2 & a_3 & a_1 \\ x_2 & x_3 & x_1 \\ y_2 & y_3 & \end{matrix}\right)
  + \Omega\left(\begin{matrix} a_3 & a_1 & a_2 \\ x_3 & x_1 & x_2 \\ y_3 & y_1 & \end{matrix}\right) \\
  & \qquad = \beta\left(0,y_2\right)\beta\left(0,y_3\right)+\beta\left(0,y_1\right)\beta\left(0,y_3\right)+\beta\left(0,y_1\right)\beta\left(0,y_2\right) \\
  & \qquad \qquad + \sum_{ \substack{h_1, h_2, h_3 \\ (h_1, h_2 , h_3) \ne (0,0,0)} } \sum_{k=1}^{3} \ \prod_{\substack{i=1\\ i\not=k}}^{3}\ \beta\left(\frac{h_k }{a_k}-\frac{h_i }{a_i},y_i\right) ,
\end{align*}
where the last sum is over all triples $\left(h_1 \bmod a_1, h_2 \bmod a_2, h_3 \bmod a_3\right) \ne \left(0,0,0\right)$.
This term vanishes for the same reasons as in the Section \ref{proofsection}, since the crucial assumption $\frac{h_k }{a_k}-\frac{h_i }{a_i} \notin \Z$ holds.

For the remaining terms we use the cotangent identity
\[
  \cot{\left(\alpha\right)}+\cot{\left(\beta\right)}=\frac{\cot{\left(\alpha\right)}\cot{\left(\beta\right)}-1}{\cot{\left(\alpha+\beta\right)}}
\]
and note that, by the definition of $\cot{y}$,
\[
  \beta(0,y)=-i \cot{\frac{y}{2i}} \, .
\]
Let $y_k^*:=\frac{y_k}{2i}$. Then
\begin{align*}
  & \Omega\left(\begin{matrix} a_1 & a_2 & a_3 \\ x_1 & x_2 & x_3 \\ y_1 & y_2 & \end{matrix}\right)
  + \Omega\left(\begin{matrix} a_2 & a_3 & a_1 \\ x_2 & x_3 & x_1 \\ y_2 & y_3 & \end{matrix}\right)
  + \Omega\left(\begin{matrix} a_3 & a_1 & a_2 \\ x_3 & x_1 & x_2 \\ y_3 & y_1 & \end{matrix}\right) \\
  & \qquad = \beta\left(0,y_2\right)\beta\left(0,y_3\right)+\beta\left(0,y_1\right)\beta\left(0,y_3\right)+\beta\left(0,y_1\right)\beta\left(0,y_2\right) \\
  & \qquad =-\frac{1}{4}\left(\cot{y_2^*}\cot{y_3^*}+\cot{y_1^*}\cot{y_3^*}+\cot{y_1^*}\cot{y_2^*}\right) \\
  & \qquad =-\frac{1}{4}\left(\cot{y_2^*}\left(\frac{\cot{y_1^*}\cot{y_3^*}-1}{\cot{y_1^*+y_3^*}}\right)+\cot{y_1^*}\cot{y_3^*}\right).
\end{align*}
By assumption, $y_1+y_2+y_3=0$, so
\[
  \cot{\left(y_1^*+y_3^*\right)}=\cot{\left(-y_2^*\right)}=-\cot{y_2^*} \, ,
\]
which yields
\[
    \Omega\left(\begin{matrix} a_1 & a_2 & a_3 \\ x_1 & x_2 & x_3 \\ y_1 & y_2 & \end{matrix}\right)
  + \Omega\left(\begin{matrix} a_2 & a_3 & a_1 \\ x_2 & x_3 & x_1 \\ y_2 & y_3 & \end{matrix}\right)
  + \Omega\left(\begin{matrix} a_3 & a_1 & a_2 \\ x_3 & x_1 & x_2 \\ y_3 & y_1 & \end{matrix}\right)
  =-\frac{1}{4} \, .
\]

\vspace{6pt}
\noindent
Case (ii).
Without loss of generality, we assume $(x_1,x_2)\in(a_1,a_2)\R+\Z^2$ (but not $(x_1,x_2,x_3)\in(a_1,a_2,a_3)\R+\Z^3$). Then, as in case (i),
\[
  r_1 - r_2 = \frac{h_1+\lambda a_1 + z_1}{a_1}-\frac{h_2+\lambda a_2 + z_2}{a_2} = \frac{h_1+ z_1}{a_1}-\frac{h_2 + z_2}{a_2}.
\]
Since $z_1$ and $z_2$ permute the summands over $h_1$ and $h_2$, we introduce a change of variables and let $\bar{h}_1:=h_1+z_1$, $\bar{h}_2:=h_2+z_2$, and $\bar{h}_3 := h_3$. We can rewrite the differences involving $r_3$ as 
\[
  r_i - r_3 = \left(\frac{\bar{h}_i}{a_i}+\lambda-r_3\right) = \left(\frac{\bar{h}_i}{a_i}-\tilde{r}_3\right) .
\]
Again we will split up our reciprocity sum into two parts:
\begin{align}
  & \Omega\left(\begin{matrix} a_1 & a_2 & a_3 \\ x_1 & x_2 & x_3 \\ y_1 & y_2 & \end{matrix}\right)
  + \Omega\left(\begin{matrix} a_2 & a_3 & a_1 \\ x_2 & x_3 & x_1 \\ y_2 & y_3 & \end{matrix}\right)
  + \Omega\left(\begin{matrix} a_3 & a_1 & a_2 \\ x_3 & x_1 & x_2 \\ y_3 & y_1 & \end{matrix}\right) \nonumber \\
  & \qquad = \sum_{\bar{h}_3 \bmod{a_3}} \beta\left(0,y_2\right)\beta\left(-\tilde{r}_3,y_3\right)+\beta\left(0,y_1\right)\beta\left(-\tilde{r}_3,y_3\right)+\beta\left(\tilde{r}_3,y_1\right)\beta\left(\tilde{r}_3,y_2\right) \nonumber \\
  & \qquad \qquad + \sum_{\substack{\bar{h}_1, \bar{h}_2, \bar{h}_3\\(\bar{h}_1,\bar{h}_2) \ne (0,0)} } \sum_{k=1}^{3} \prod_{\substack{i=1\\ i\not=k}}^{3}\beta\left(\frac{\bar{h}_k }{a_k}-\frac{\bar{h}_i }{a_i},y_i\right) , \label{2partspliteq}
\end{align}
where the last sum is over all triples $\left(\bar{h}_1 \bmod a_1, \bar{h}_2 \bmod a_2, \bar{h}_3 \bmod a_3\right)$ such that $(\bar{h}_1,\bar{h}_2) \ne (0,0)$.
As before, the last term in \eqref{2partspliteq} vanishes for the same reasons as in Section \ref{proofsection}. Thus
\begin{align*}
  & \Omega\left(\begin{matrix} a_1 & a_2 & a_3 \\ x_1 & x_2 & x_3 \\ y_1 & y_2 & \end{matrix}\right)
  + \Omega\left(\begin{matrix} a_2 & a_3 & a_1 \\ x_2 & x_3 & x_1 \\ y_2 & y_3 & \end{matrix}\right)
  + \Omega\left(\begin{matrix} a_3 & a_1 & a_2 \\ x_3 & x_1 & x_2 \\ y_3 & y_1 & \end{matrix}\right) \\
  & \qquad = \sum_{\bar{h}_3 \bmod{a_3}} \beta\left(0,y_2\right)\beta\left(-\tilde{r}_3,y_3\right)+\beta\left(0,y_1\right)\beta\left(-\tilde{r}_3,y_3\right)+\beta\left(\tilde{r}_3,y_1\right)\beta\left(\tilde{r}_3,y_2\right) \\
  & \qquad = \sum_{\bar{h}_3 \bmod{a_3}} \frac{1}{2}\frac{e^{y_2}+1}{e^{y_2}-1} \frac{e^{\{-\tilde{r}_3\}y_3}}{e^{y_3}-1}+\frac{1}{2}\frac{e^{y_1}+1}{e^{y_1}-1} \frac{e^{\{-\tilde{r}_3\}y_3}}{e^{y_3}-1}+\frac{e^{\{\tilde{r}_3\}y_1}}{e^{y_1}-1} \frac{e^{\{\tilde{r}_3\}y_2}}{e^{y_2}-1} \, .
\end{align*}
After bringing the fractions onto a common denominator, we obtain the numerator
\begin{align*}
  & e^{y_1+y_2+\{-\tilde{r}_3\}y_3}-e^{y_2+\{-\tilde{r}_3\}y_3}+e^{y_1+\{-\tilde{r}_3\}y_3}-e^{\{-\tilde{r}_3\}y_3}+e^{y_1+y_2+\{-\tilde{r}_3\}y_3}-e^{y_1+\{-\tilde{r}_3\}y_3}+e^{y_2+\{-\tilde{r}_3\}y_3} \\
  &\quad -e^{\{-\tilde{r}_3\}y_3}+2e^{\{\tilde{r}_3\}y_1+\{\tilde{r}_3\}y_2+y_3}-2e^{\{\tilde{r}_3\}y_1+\{\tilde{r}_3\}y_2} \\
  &= e^{\{\tilde{r}_3\}y_1+\{\tilde{r}_3\}y_2}-e^{-\{-\tilde{r}_3\}y_1+\{\tilde{r}_3\}y_2}+e^{\{\tilde{r}_3\}y_1-\{-\tilde{r}_3\}y_2}-e^{-\{-\tilde{r}_3\}y_1-\{-\tilde{r}_3\}y_2}+e^{\{\tilde{r}_3\}y_1+\{\tilde{r}_3\}y_2} \\
  &\quad -e^{\{\tilde{r}_3\}y_1-\{-\tilde{r}_3\}y_2}+e^{-\{-\tilde{r}_3\}y_1+\{\tilde{r}_3\}y_2}-e^{-\{-\tilde{r}_3\}y_1-\{-\tilde{r}_3\}y_2}+2e^{-\{-\tilde{r}_3\}y_1-\{-\tilde{r}_3\}y_2}-2e^{\{\tilde{r}_3\}y_1+\{\tilde{r}_3\}y_2} \\
  &=0 \, . \qedhere
\end{align*}
\end{proof}


\section{Final Remarks}\label{finalsection}

Trying to extend Hall--Wilson--Zagier's Theorem \ref{hwzreciprocity} to the next case of four sets of variables, the task is to study the reciprocity sum
\[
    \Omega\left(\begin{matrix} a_1 & a_2 & a_3 & a_4 \\ x_1 & x_2 & x_3 & x_4 \\ y_1 & y_2 & y_3 \end{matrix}\right)
  + \Omega\left(\begin{matrix} a_2 & a_3 & a_4 & a_1 \\ x_2 & x_3 & x_4 & x_1 \\ y_2 & y_3 & y_4 \end{matrix}\right)
  + \Omega\left(\begin{matrix} a_3 & a_4 & a_1 & a_2 \\ x_3 & x_4 & x_1 & x_2 \\ y_3 & y_4 & y_1 \end{matrix}\right)
  + \Omega\left(\begin{matrix} a_4 & a_1 & a_2 & a_3 \\ x_4 & x_1 & x_2 & x_3 \\ y_4 & y_1 & y_2 \end{matrix}\right)
\]
in the four cases
\begin{enumerate}[(i)]
\item $(x_1,x_2,x_3,x_4)\in(a_1,a_2,a_3,a_4)\R+\Z^4$;
\item $(x_i,x_j,x_k)\in(a_i,a_j,a_k)\R+\Z^3$ for some $1\leq i <j < k \leq 4$ but not (i);
\item $(x_i,x_j)\in(a_i,a_j)\R+\Z^2$ for some $1\leq i <j \leq 4$ but not (i) or (ii);
\item none of the above.
\end{enumerate}

\vspace{6pt}
\noindent
Case (i) is covered by \cite{bayadraouj} (and can be easily recovered with a calculation similar to that in the last section); in this case the reciprocity sum equals
\[
  \frac{i}{8}\left(\cot{\left(\frac{y_0}{2i}\right)}+\cot{\left(\frac{y_1}{2i}\right)}+ \cot{\left(\frac{y_2}{2i}\right)}+\cot{\left(\frac{y_3}{2i}\right)}\right) .
\]

\vspace{6pt}
\noindent
Case (iv) is covered by Theorem~\ref{multivariatereciprocitytheorem}.

\vspace{6pt}
\noindent
For Case (iii), a calculation similar to that in the last section reveals that the reciprocity sum vanishes.

\vspace{6pt}
\noindent
For Case (ii), similar calculations yield
\begin{align*}
  & \Omega\left(\begin{matrix} a_1 & a_2 & a_3 & a_4 \\ x_1 & x_2 & x_3 & x_4 \\ y_1 & y_2 & y_3 \end{matrix}\right)
  + \Omega\left(\begin{matrix} a_2 & a_3 & a_4 & a_1 \\ x_2 & x_3 & x_4 & x_1 \\ y_2 & y_3 & y_4 \end{matrix}\right)
  + \Omega\left(\begin{matrix} a_3 & a_4 & a_1 & a_2 \\ x_3 & x_4 & x_1 & x_2 \\ y_3 & y_4 & y_1 \end{matrix}\right)
  + \Omega\left(\begin{matrix} a_4 & a_1 & a_2 & a_3 \\ x_4 & x_1 & x_2 & x_3 \\ y_4 & y_1 & y_2 \end{matrix}\right) \\
  &= \sum_{h_4 \bmod{a_4}} \frac{ 
    \left(\begin{matrix}e^{\{\tilde{r}_4\}y_1+\{\tilde{r}_4\}y_2 -\{-\tilde{r}_4\}y_3}
-e^{-\{-\tilde{r}_4\}y_1-\{-\tilde{r}_4\}y_2+\{\tilde{r}_4\}y_3}
+e^{\{\tilde{r}_4\}y_1-\{-\tilde{r}_4\}y_2+\{\tilde{r}_4\}y_3}\\
-e^{-\{-\tilde{r}_4\}y_1+\{\tilde{r}_4\}y_2-\{-\tilde{r}_4\}y_3}
+2e^{-\{-\tilde{r}_4\}y_1+\{\tilde{r}_4\}y_2+\{\tilde{r}_4\}y_3}  
-2e^{\{\tilde{r}_4\}y_1-\{-\tilde{r}_4\}y_2-\{-\tilde{r}_4\}y_3}\end{matrix}\right)
    }{ 4 \left( e^{ y_1 } - 1 \right) \left( e^{ y_2 } - 1 \right) \left( e^{ y_3 } - 1 \right) \left( e^{ y_4 } - 1 \right) } \, ,
\end{align*}
where $\tilde r_4$ is defined analogously to the way we defined $\tilde r_3$ in the previous section.
As with all the previous summands, this final sum exhibits an intriguing symmetry, but it is not clear to us if it vanishes or evaluates to a simple expression.
As mentioned earlier, for higher-dimensional Bernoulli--Dedekind sums (i.e., for larger $n$), there are more intermediate cases, in which we have a partial linear relation such as $(x_1,x_2,x_3)\in(a_1,a_2,a_3)\R+\Z^3$, and it is not clear to us how one can easily deal with them.

We conclude with one more open problem, namely, that of the computational complexity of Bernoulli--Dedekind sums.
Any Dedekind-like sum that obeys a two-term reciprocity law is instantly computable through the Euclidean algorithm.
However, the computational complexity of ``higher-dimensional" Dedekind-like sums is more subtle. It was proved in \cite{cotangent} that the cotangent-generalizations of the Dedekind sum are polynomial-time computable (in the input length of the integer parameters). It is not clear to us how the argument in \cite{cotangent} could be modified to say anything about the complexity of Bernoulli--Dedekind sums.


\bibliographystyle{amsplain}

\def\cprime{$'$} \def\cprime{$'$}
\providecommand{\bysame}{\leavevmode\hbox to3em{\hrulefill}\thinspace}
\providecommand{\MR}{\relax\ifhmode\unskip\space\fi MR }
\providecommand{\MRhref}[2]{%
  \href{http://www.ams.org/mathscinet-getitem?mr=#1}{#2}
}
\providecommand{\href}[2]{#2}

\setlength{\parskip}{0cm} 
\end{document}